\documentclass[preprint, 12pt]{elsarticle}
\usepackage{amsthm}

\usepackage{amssymb}
\newtheorem{thm}{Theorem}
\newtheorem{theorem}{Theorem}[section]
\newtheorem{corollary}[theorem]{Corollary}

\newtheorem{lemma}[theorem]{Lemma}

\newcommand \bpig {{\textup B}_{\pi}(G)}

\newcommand \ipi {{\textup I}_{\pi}}
\newcommand \ipig {{\textup I}_{\pi}(G)}

\newcommand \irr {\textup{Irr}}
\newcommand \irrg {\textup{Irr}(G)}

\newcommand \ibrg {{\textup{IBr}}_p(G)}

\newcommand \ngp {{\bf{N}}_G(P)}

\newcommand \ngq {{\bf{N}}_G(Q)}
\newcommand \ngqd {{\bf{N}}_G(Q, \delta)}

\newcommand \dvd {\hbox {\big|}}
\newcommand \ndvd {\hbox {/}\kern-5pt\dvd}
\newcommand \nrml {\lhd}
\def \< {\langle}
\def \> {\rangle}

\newcommand \ntq {{\bf{N}}_T(Q)}

\newcommand \opi {{\bf{O}}_{\pi}}
\newcommand \opig {{\bf{O}}_{\pi}(G)}
\newcommand \opid {{\bf{O}}_{\pi'}}

\def\cent#1#2{{\bf C}_{#1}(#2)}

\def\ker#1{{\rm ker} (#1)}

\def\norm#1#2{{\bf{N}}_{#1} (#2)}

\def\B#1#2{{\rm B}_{#1} (#2)}
\def\Bpi#1{\B {\pi}{#1}}

\def\I#1#2{{\rm I}_{#1} (#2)}

\def\phi{\varphi}

\journal{Journal of Algebra}

\begin{document}

\begin{frontmatter}



\title{Counting lifts of Brauer characters}


\author{James P. Cossey}

\address{Department of Theoretical and Applied Mathematics, University of Akron, Akron, OH 44325}
\ead{cossey@uakron.edu}

\author{Mark L. Lewis}

\address{Department of Mathematical Sciences, Kent State University, Kent, OH 44242}
\ead{lewis@math.kent.edu}

\begin{abstract}
In this paper we examine the behavior of lifts of Brauer characters in solvable groups.  In the main result, we show that if $\varphi \in \ibrg$ is a Brauer character of a solvable group such that $\varphi$ has an abelian vertex subgroup $Q$, then the number of lifts of $\varphi$ in $\irrg$ is at most $|Q|$.  In order to accomplish this, we develop several results about lifts of Brauer characters in solvable groups that were previously only known to be true in the case of groups of odd order.

\end{abstract}


\end{frontmatter}

%This set of notes is joint work of J.P. Cossey and Mark L. %Lewis. In
%the first section we show that J. P.'s work on generalized %vertices
%on lifts in groups of odd order can be extended to lifts of %solvable
%groups where $2 \in \pi$.  In the second section, we use this %to
%bound the number of lifts of $\pi$-partial characters where $2 %\in
%\pi$, the group is solvable, and the vertex is abelian.

\section{Introduction}

Let $G$ be a group, and let $p$ be a prime.  We write $\ibrg$ for the set of irreducible $p$-Brauer characters of $G$.  Let $G^o$ be the $p$-regular elements of $G$.  If $\chi$ is a character of $G$, then we use $\chi^o$ to denote the restriction of $\chi$ to $G^o$.  Given a $p$-Brauer character $\varphi \in \ibrg$, we say that $\chi \in \irrg$ is a {\it lift} of $\varphi$ if $\chi^o = \varphi$.  When $G$ is $p$-solvable, the Fong-Swan theorem shows that $\varphi$ has a lift.  Much of the study of lifts has focused on particular canonical sets of lifts \cite{pisep}.  The first author has initiated a study of all lifts of $\varphi$.  For example, when $|G|$ is odd, he has shown that the number of lifts of $\varphi$ can be bounded in terms of a vertex subgroup for $\varphi$.

In a $p$-solvable group, we say $Q$ is a {\it vertex} for $\varphi$ if there is a subgroup $U$ of $G$ so that $\varphi$ is induced from a $p$-Brauer character of $U$ having $p'$-degree and $Q$ is a Sylow $p$-subgroup of $U$.  It is known \cite{IsNa} that all of the vertices for $\varphi$ are conjugate in $G$.  In \cite{bounds}, the first author showed that if $|G|$ is odd and $Q$ is a vertex for $\varphi$, then the number of lifts of $\varphi$ is at most $|Q:Q'|$.

In this paper, we wish to remove the hypothesis that $|G|$ is odd.  However, we do need an oddness hypothesis for our work, and so, we assume that $G$ is $p$-solvable where $p$ is an odd prime.  We will present an example that justifies this oddness hypothesis.  Also, at this time, we need to add the hypothesis that $Q$ be abelian.  We do not have any examples justifying this hypothesis, and in fact, we would not be surprised if this can be removed.  We will discuss one possible way to remove this hypothesis.  In particular, we prove the following.

\begin{thm} \label{thmA}
Let $G$ be a $p$-solvable group and let $p$ be an odd prime.  If $\varphi \in \ibrg$ has abelian vertex subgroup $Q$, then the number of lifts of $\varphi$ is at most $|Q|$.
\end{thm}

To prove this theorem, we consider the generalized vertices defined by in \cite{vertodd}.  To do this, we need the theory of $p$-factored characters (see \cite{pisep} for much more on factored characters).  A character is {\it $p$-factored} if it is the product of a $p$-special character and a $p'$-special character.  Let $\chi \in \irrg$.  Then $(Q,\delta)$ is a {\it generalized vertex} for $\chi$ if there is a subgroup $U$ with a $p$-factored character $\psi \in \irr(U)$ and Sylow $p$-subgroup $Q$ of $U$ so that $\psi^G = \chi$ and $\delta$ is the restriction to $Q$ of the $p$-special factor of $\psi$.

In general, there is little that one can say about the set of generalized vertices for a character $\chi \in \irrg$.  However, when $|G|$ is odd and $\chi$ is a lift of a Brauer character with generalized vertex $(Q,\delta)$, the first author proved in \cite{vertodd} that $\delta$ is linear and all the generalized vertices for $\chi$ are all conjugate to $(Q,\delta)$.

Again, we are interested in this paper in removing the hypothesis that $|G|$ is odd.  Again, we will need to replace these with the hypotheses that $G$ is $p$-solvable and $p$ is odd (though we do not require here that the vertex subgroup $Q$ is abelian). First, we use a theorem of Navarro \cite{vertodd} to see that under these hypotheses, if $\chi$ is the lift of a Brauer character with vertex $(Q,\delta)$, then $\delta$ is linear.  With this result in hand, we are able to prove the following.

\begin{thm}\label{thmB}
Let $G$ be a $p$-solvable group, and let $p$ be an odd prime.  If $\chi \in \irrg$ is the lift of a $p$-Brauer character, then all of the generalized vertices for $\chi$ are conjugate.
\end{thm}

For a subgroup $Q \subseteq G$ and a character $\delta$ of $Q$, we will let $\ngqd$ denote the stabilizer of $\delta$ in $\ngq$.  The key step in proving Theorem \ref{thmA} is the following theorem where the number of lifts that have $(Q,\delta)$ is bounded.  A version of this theorem was also the key step in proving the upper bound on the number of lifts in the case that $G$ has odd order in \cite{bounds}.  Our approach in this paper to prove this theorem is very different from the approach used in \cite{bounds}.

\begin{thm}
Let $G$ be a $p$-solvable group, and let $p$ be an odd prime.  Suppose $\varphi \in \ibrg$ has an abelian vertex subgroup $Q$.  If $\delta \in \irr(Q)$, then the number of lifts of $\varphi$ having generalized vertex $(Q,\delta)$ is at most $|\ngq: \ngqd|$.
\end{thm}

We note that if $p = 2$, one can find in ${\rm GL}_2 (3)$ a lift of a Brauer character whose vertex is not linear.

\section{Generalized vertices}

Rather than work with Brauer characters, we work in the context of Isaacs' partial characters.  Hence, we will have a set of primes $\pi$.  To define the $\pi$-partial characters, one needs to assume that $G$ is $\pi$-separable.  As in the context of Brauer characters, we let $G^o$ denote the set of $\pi$-elements in $G$.  Given an ordinary character $\chi$, we use $\chi^o$ to denote the restriction of $\chi$ to $G^o$.  The $\pi$-partial characters of $G$ are the functions defined on $G^o$ that are restrictions of ordinary characters.  The $\pi$-partial characters that cannot be written as the sum of two other partial characters are called irreducible.  We use $\ipig$ to denote the irreducible $\pi$-partial characters of $G$.  For a full exposition on $\pi$-partial characters, we refer the reader to \cite{pisep} and \cite{pipart}.

The irreducible $\pi$-partial characters of $G$ have many properties in common with the irreducible Brauer characters of a $p$-solvable group.  In fact, if $\pi = p'$, then $\ipig = \ibrg$, and the requirement that $p$ is odd is equivalent
to $2 \in \pi$.  Therefore, we have the assumption that $2 \in \pi$.  For example, we can define induction of partial characters from subgroups.  Given an irreducible $\pi$-partial character $\phi$ of $G$, we can define a vertex $Q$ for $\phi$ to be a Hall $\pi'$-subgroup of a subgroup $U$ that contains a $\pi$-partial character $\kappa$ of $\pi$-degree that induces $\phi$.  Isaaacs and Navarro have proved in \cite{IsNa} that all of the vertices for $\phi$ are conjugate in $G$.  (A different proof of this fact is in \cite{istan}.)  There also exists a Clifford correspondence for $\pi$-partial characters.  If $G$ is $\pi$-separable and $N \nrml G$ and $\theta \in \ipi(N)$, then induction is a bijection from the set $\ipi(G_{\theta} \mid \theta)$ to $\ipi(G \mid \theta)$ (see \cite{Fong}).

We also need to consider $\pi$-special characters.  Let $G$ be a $\pi$-separable group.  A character $\chi \in \irr G$ is $\pi$-special if $\chi (1)$ is a $\pi$-number and for every subnormal group $M$ of $G$, the irreducible constituents of $\chi_M$ have $\pi$-order.  Many of the basic results of $\pi$-special characters can be found in Section 40 of \cite{hupte} and Chapter VI of \cite{Mawo}.  One result that is proved is that if $\alpha$ is $\pi$-special and $\beta$ is $\pi'$-special, then $\alpha \beta$ is necessarily irreducible. Furthermore, if $\alpha'$ is $\pi$-special and $\beta'$ is $\pi$-special so that $\alpha' \beta' = \alpha \beta$, then $\alpha' = \alpha$ and $\beta' = \beta$.  We say that $\chi$ is $\pi$-factored (or factored, if the $\pi$ is clear from context) if $\chi = \alpha \beta$.  Another result is that if $H$ is a Hall $\pi$-subgroup of $G$, then restriction defines an injection from the $\pi$-special characters of $G$ into $\irr (H)$.

%The goal of this section is to prove that Cossey's results %regarding generalized vertices are true when $2 \in \pi$.
Following the terminology introduced in \cite{vertodd}, we say $(Q,\delta)$ is a {\it generalized $\pi$-vertex} for $\chi \in \irrg$ if there exists a pair $(U,\psi)$ (where $U \subseteq G$ and $\psi \in \irr(U)$) so that $\psi^G = \chi$, $Q$ is a Hall $\pi$-complement of $U$, $\psi = \alpha \beta$ where $\alpha$ is $\pi$-special and $\beta$ is $\pi'$-special, and $\beta_Q = \delta$.  In this context, we say that $(U,\psi)$ is a {\it generalized $\pi$-nucleus} for $\chi$.

In \cite{vertodd}, the first author proved that if $|G|$ is odd and $\chi \in \irrg$ with $\chi^o \in \ipig$, then the generalized $\pi$-vertices for $\chi$ are conjugate. We now show that the hypothesis that $|G|$ is odd can be replaced by the hypothesis that $G$ is $\pi$-separable and $2 \in \pi$.  Our argument will parallel the argument in \cite{vertodd}.

The main result, which is the $\pi$-version of Theorem \ref{thmB} is the following.

\begin{theorem}\label{second main}
Let $\pi$ be a set of primes with $2 \in \pi$, and let $G$ be a $\pi$-separable group. If $\chi \in \irrg$ with $\chi^o \in \ipig$, then all the generalized $\pi$-vertices for $\chi$ are conjugate.
\end{theorem}

The key to our work is a recent result of Navarro.  Replacing $p$ by a set of primes $\pi$ with $2 \in \pi$, the proof of Lemma 2.1 of \cite{newnav} proves:

\begin{lemma}\label{navarro}
Let $\pi$ be a set of primes with $2 \in \pi$, and let $G$ be a $\pi$-separable group.  Let $\chi \in \irrg$ be $\pi'$-special.  If $\chi (1) > 1$, then $\chi^o$ is not in $\ipig$.
\end{lemma}

For the remainder of this section, our work will parallel the work in \cite{vertodd}.  The following should be compared with Lemma 2.3 of \cite{vertodd}.

\begin{lemma} \label{lemma 2.3}
Let $\pi$ be a set of primes with $2 \in \pi$, and let $G$ be a $\pi$-separable group.  Let $\chi \in \irrg$ with $\chi^o \in \ipig$.  If $U \le G$ and $\psi \in \irr(U)$ is a $\pi$-factored character that induces $\chi$, then the $\pi'$-special factor of $\psi$ is linear.  Moreover, if $Q$ is a Hall $\pi$-complement of $U$, then $Q$ is a vertex subgroup of $\chi^o$.
\end{lemma}

\begin{proof}
Note that since $\chi^o \in \ipig$, and $\psi^G = \chi$, then
$\psi^o \in \ipi(U)$.  Since $\psi$ is $\pi$-factored, we have $\psi = \alpha \beta$ where $\alpha$ is $\pi$-special and $\beta$ is $\pi'$-special.  It follows that $\beta^o \in \ipi(U)$.  By Lemma \ref{navarro}, $\beta (1) = 1$.  It follows that $\psi$ has $\pi$-degree and $\psi^o \in \ipi(U)$.  By Theorem B of \cite{IsNa}, $Q$ is a vertex subgroup of $\chi^o$.
\end{proof}

The next lemma is similar to Lemma 3.1 of \cite{vertodd}.

\begin{lemma} \label{lemma 3.1}
Let $\pi$ be a set of primes with $2 \in \pi$, and let $G$ be a $\pi$-separable group.  Let $\chi \in \irrg$ with $\chi = \alpha \beta$ where $\alpha$ is $\pi$-special and $\beta$ is linear and $\pi'$-special.  Suppose $\psi \in \irr(U)$ is $\pi$-factored and induces $\chi$.  If $\delta$ is the $\pi'$-special factor of $\psi$, then $\beta_U = \delta$.
\end{lemma}

\begin{proof}
Note that  $\alpha = \alpha \beta \beta^{-1} = \chi \beta^{-1}$.  It
follows that $(\psi \beta^{-1}|_U)^G = \psi^G \beta^{-1} = \chi \beta^{-1} = \alpha$.  Since $\alpha$ is $\pi$-special, we may use Theorem C of \cite{induct} to see that $\psi \beta^{-1}|_U$ is $\pi$-special.  We can write $\psi = \gamma \delta$ where $\gamma$ is $\pi$-special.  Now, $\gamma^o = \psi^o = (\psi
\beta^{-1}|_U)^o$, and so, $\gamma = \psi \beta^{-1}|_U = \gamma \delta \beta^{-1}|_U$. It follows that $\delta \beta^{-1}|_U = 1_U$, and hence, $\delta = \beta_U$.
\end{proof}

The next result should be compared with Lemma 3.2 of \cite{vertodd}.  Let $\pi$ be a set of primes with $2 \in \pi$ and $G$ is $\pi$-separable.  We will need the basic properties of the set $\bpig \subseteq \irrg$ introduced in \cite{pisep}.  In particular, we need to know that restriction to $G^o$ gives a bijection from $\bpig$ to $\ipig$ and that the $\pi$-special characters of $G$ are precisely the characters of $\pi$-degree in $\bpig$.  We will also use the magic field automorphism that was described in \cite{odd}.  We write $\sigma$ to denote the magic field automorphism.  Let $\chi \in \irrg$. In \cite{odd}, it is proved that $\chi \in \Bpi G$ if and only if $\chi^\sigma = \chi$ and $\chi^o \in \ipig$.

\begin{lemma} \label{lemma 3.2}
Let $\pi$ be a set of primes with $2 \in \pi$, and let $G$ be a $\pi$-separable group with subgroup $U$.  Suppose $\chi \in \irrg$ satisfies $\chi^o \in \ipig$.  Assume $\psi \in \irr(U)$ is $\pi$-factored so that $\chi = \psi^G$.  Suppose $|G:U|$ is a $\pi$-number and the $\pi'$-special factor of $\psi$ extends to $G$. Then $\chi$ is $\pi$-factored.
\end{lemma}

%In this proof we do some restriction to pi' elements.  should we use different notation for this?

We will use the notation $\beta'$ to denote the restriction of an ordinary  character $\beta$ of $G$ to the $\pi'$-elements of $G$.

\begin{proof}
Let $\psi = \alpha \beta$ where $\alpha$ is $\pi$-special and
$\beta$ is $\pi'$-special.  Let $\phi = \beta' \in \I {\pi'}U$. Let $\eta \in \irrg$ be an extension of $\beta$.  Now, $(\eta')_U = (\eta_U)' = \beta' \in \I {\pi'}U$.  It follows that $\eta' \in \I {\pi'}G$.  Let $\delta \in \B {\pi'}G$ so that $\delta' = \eta'$. Observe that $\delta (1) = \eta (1) = \beta (1)$ is a $\pi'$-number, and so $\delta$ is $\pi'$-special.  Also, $(\delta_U)' = \beta'$ implies that $\delta_U \in \irr(U)$.  By Theorem A of \cite{induct}, $\delta_U$ is $\pi'$-special.  This implies that $\delta_U = \beta$.

We now have $\chi = \psi^G = (\alpha \beta)^G = \alpha^G \delta$.  This implies that $\alpha^G \in \irrg$.  Notice that $(\alpha^G)^\sigma = (\alpha^\sigma)^G = \alpha^G$.  Also, $\chi^o = (\alpha^G \delta)^o = (\alpha^G)^o \delta^o \in \ipig$, and so, $(\alpha^G)^o \in \ipig$.  It follows that $\alpha^G$ is $\pi$-special.  We conclude that $\chi$ is $\pi$-factored.
\end{proof}

The next result is similar to Corollary 3.3 of \cite{vertodd}.

\begin{lemma} \label{cor 3.3}
Let $\pi$ be a set of primes with $2 \in \pi$, and let $G$ be a $\pi$-separable group.  Let $\chi \in \irrg$ be $\pi$-factored and have $\pi$-degree.  Let $N$ be a normal subgroup of $G$ and suppose $\theta \in \irr(N)$ is a constituent of $\chi_N$.  Let $T$ be the stabilizer of $\theta$ in $G$.  If $\psi \in \irr(T \mid \theta)$ is the Clifford correspondent for $\chi$ with respect to $\theta$, then $\psi$ is $\pi$-factored.
\end{lemma}

\begin{proof}
Observe that $\theta$ is $\pi$-factored.  We can write $\chi =
\gamma \delta$ and $\theta = \alpha \beta$ where $\gamma$ and
$\alpha$ are $\pi$-special and $\delta$ and $\beta$ are
$\pi'$-special.  Since $\chi$ has $\pi$-degree, $\delta (1) = 1$ and thus, $\delta_N = \beta$.  It follows that $T$ is the stabilizer of $\alpha$ in $G$.  We take $\mu \in \irr(T \mid \alpha)$ to be the Clifford correspondent for $\gamma$ with respect to $\alpha$. We have $\gamma (1) = |G:T| \mu (1)$, and thus, $\mu (1)$ is a $\pi$-number. Observe that $(\mu^o)^G = (\mu^G)^o = \gamma^o \in \ipig$ and thus, $\mu^o \in \ipi(T)$.  Since $\alpha^\sigma = \alpha$, we have $\mu^\sigma \in \irr(T \mid \alpha)$.  Since $(\mu^\sigma)^G = (\mu^G)^\sigma = \mu^G$, it follows that $\mu^\sigma = \mu$, and we conclude that $\mu$ is $\pi$-special.  Because $\delta$ is linear and $\pi'$-special, $\delta_T$ is $\pi'$-special.  We see that $(\mu\delta_T)^G = \mu^G \delta = \gamma \delta= \chi$.  Also, $(\mu\delta_T)_N = \mu_N \delta_N$, and so, $\alpha \beta = \theta$ is a constituent of $(\mu\delta_T)_N$.  We obtain $\mu\delta_T \in \irr(T \mid \theta)$.  Since $(\mu\delta_T)^G = \chi = \psi^G$, we can use the Clifford correspondence to see that $\psi = \mu \delta_T$.  Therefore, $\psi$ is $\pi$-factored.
\end{proof}

Since the proof of the next lemma is essentially the proof of Lemma 3.4 of \cite{vertodd} where Lemma \ref{lemma 2.3} is used in place of Lemma 2.3 of \cite{vertodd}, we do not include it here.

\begin{lemma} \label{lemma 3.4}
Let $\pi$ be a set of primes with $2 \in \pi$, and let $G$ be a $\pi$-separable group.  Let $\chi \in \irrg$ be a lift of $\phi \in \ipig$, and suppose $N$ is normal in $G$ such that the constituents of $\chi_N$ are $\pi$-factored.  Suppose $\psi \in \irr(U)$ is $\pi$-factored, and suppose $\psi^G = \chi$.  Then $|NU:U|$ is a $\pi$-number.
\end{lemma}

%Again, the proof of the next lemma is essentially the proof of %Lemma 3.5 of \cite{vertodd} where Lemma \ref{lemma 3.1} is %used in place of Lemma 3.1 of \cite{vertodd} and Lemma %\ref{lemma 3.4} is used in place of Lemma 3.4 of %\cite{vertodd}, so we do not include it here.  At this time, %we need to assume solvability for one place in this lemma.  In %particular, the solvability of $G$ is needed to obtain the %fact that $N/K$ is an abelian chief factor which is used to %show that $K(H \cap N)$ is normal in $N$ and hence normal in %$G$.   We would like to know if we can remove the solvability %hypothesis and still obtain the conclusion.

This next lemma is similar to Lemma 3.5 of \cite{vertodd}.

\begin{lemma} \label{lemma 3.5}
Let $\pi$ be a set of primes with $2 \in \pi$, and let $G$ be a $\pi$-separable group.  Let $\chi \in \irrg$ be a lift of $\phi \in \ipig$.  Suppose $\psi$ is a $\pi$-factored character of some subgroup $H$ of $G$ that induces $\chi$, and suppose there is a normal subgroup $N$ of $G$ such that the constituents of $\chi_N$ are $\pi$-factored and $G = NH$.  Then $\chi$ is $\pi$-factored and the $\pi'$-special factor of $\chi$ restricts irreducibly to the $\pi'$-special factor of $\psi$.
\end{lemma}

\begin{proof}
Notice that the second conclusion follows from the first conclusion by Lemma \ref{lemma 3.1}.  We assume the first conclusion is not true, and we take $G$, $N$, and $H$ to be a counterexample with $|G:H| + |N|$ minimal.

By Lemma \ref{lemma 2.3}, the $\pi'$-special factor of $\psi$ is linear, so $\psi (1)$ is a $\pi$-number.  Applying Lemma \ref{lemma 3.4}, we see that $|G:H| = |HN:N|$ is a $\pi$-number.  Since $\chi (1) = |G:H| \psi (1)$, we see that $\chi$ has $\pi$-degree.

Choose $K$ normal in $G$ so that $N/K$ is a chief factor for $G$.  Notice that the irreducible constituents of $\chi_K$ are $\pi$-factored.  If $G = HK$, then $G$, $K$, and $H$ form a counterexample with $|G:H| + |K| < |G:H| + |N|$ violating the choice of minimal counterexample.  Thus, we have $HK < G$.

Notice that $G = NH = N(HK)$.  Notice that $\psi^{HK} \in \irr(HK)$ will be a lift of a partial character in $\ipi(HK)$.  Also, the irreducible constituents of $(\psi^{HK})_K$ are constituents of $\chi_K$, and thus must be factored.  If $H < HK$, then $|HK:H| + |K| < |G:H| + |N|$, and so $HK$, $K$, and $H$ cannot form a counterexample.  Thus, $\psi^{HK}$ must be factored and induce $\chi$.  Also, $|G:HK| + |N| < |G:H| + |N|$, so $G$, $HK$, and $N$ do not form a counterexample.  We conclude that $\chi$ is $\pi$-factored, a contradiction.  This implies that $H = HK$.

We have $K \le H$.  Let $\eta$ be an irreducible constituent of $\psi_K$.  Notice that $\eta^N$ has an irreducible constituent $\theta$ which is a constituent of $\chi_N$, so $\theta$ and $\eta$ are both $\pi$-factored.  Since $\chi$ has $\pi$-degree, $\theta$ has a linear $\pi'$-special factor.  If $\nu$ is the $\pi'$-special factor of $\eta$, then $\nu$ extends to both the $\pi'$-special factor of $\theta$ and the $\pi'$-special factor of $\psi$.  This implies that $\nu$ is invariant in both $N$ and $H$.  Since $G = NH$, we conclude that $\nu$ is $G$-invariant.

Note that $|N : K|$ divides the $\pi$-number $|G: H|$ and thus $|N : K|$ is a $\pi$-number.  Let $\hat\nu$ be the unique $\pi'$-special extension of $\nu$ to $N$, and since $\nu$ is $G$-invariant so is $\hat\nu$.  We can now apply Corollary 4.2 of \cite{pisep} to see that restriction defines a bijection from $\irr(G \mid \hat\nu)$ to $\irr(H \mid \hat\nu_{N \cap H})$.  Observe that the $\pi'$-special factor of $\psi$ will belong to $\irr(H \mid \hat\nu_{N \cap H})$ since $\hat\nu_{N \cap H}$ is the unique $\pi'$-special extension of $\nu$ to $N \cap H$.  It follows that the $\pi'$-special factor of $\psi$ extends to $G$, and applying Lemma \ref{lemma 3.2} we conclude that $\chi$ is factored, as desired.
\end{proof}

We make use of the normal nucleus constructed by Navarro in \cite{navvert}.  We quickly summarize this construction.  Fix a character $\chi \in \irrg$.  Navarro shows that there is a unique subgroup $N$ that is maximal subject to being normal in $G$ and the irreducible constituents of $\chi_N$ are $\pi$-factored.  If $N = G$, then take $(G,\chi)$ to be the normal nucleus of $\chi$.  If $N < G$, let $\theta$ be an irreducible constituent of $\chi_N$.  Navarro shows that $\theta$ is not $G$-invariant.  Then we let $\chi_{\theta} \in \irr(G_{\theta} \mid \theta)$ be the Clifford correspondent for $\chi$ with respect to $\theta$ (see Theorem 6.11 of \cite{text} or Theorem 19.6 (d) of \cite{hupte}).  We define the normal nucleus for $\chi$ to be the normal nucleus of $\chi_\theta$ which can be computed inductively since $G_{\theta} < G$.  It can be easily seen that all of the normal nuclei for $\chi$ are conjugate.

The proof of Theorem \ref{second main} is essentially the proof Theorem 4.1 of \cite{vertodd}, and thus we do not include it here in full detail.  However, we do provide a brief sketch of the proof.  The goal is to show that if $(U, \psi)$ is any generalized $\pi$-nucleus of $\chi$, then the generalized $\pi$-vertex of $\chi$ defined by $(U, \psi)$ is conjugate to a vertex for $\chi$ arising from a normal nucleus.  Lemmas \ref{lemma 2.3} and \ref{lemma 3.1} allow us to assume that $\chi$ is not factorable.  Let $N \nrml G$ be maximal so that the constituents of $\chi_N$ are factorable.  By Lemma \ref{lemma 3.4}, we see that $|NU : U|$ is a $\pi$-number, and thus Lemma \ref{lemma 3.5} allows us to replace the pair $(U, \psi)$ with $(NU, \psi^{NU})$, and thus we may assume $N \subseteq U$.  Letting $\theta$ be a constituent of $\psi_N$, we use Lemma \ref{cor 3.3} and Lemma \ref{lemma 3.1} to replace $(U, \psi)$ with the pair $(U_{\theta}, \xi)$, where $\xi$ is the Clifford correspondent for $\psi$ in $\irr(U_{\theta} |\theta)$.  We finish by applying the inductive hypothesis to the group $G_{\theta}$ and the Clifford correspondent for $\chi$ lying over $\theta$, which by definition has a normal nucleus in common with $\chi$.

% where Lemma \ref{lemma 2.3} replaces Lemma 2.3 of \cite{vertodd}, Lemma \ref{lemma 3.1} replaces Lemma 3.1 of \cite{vertodd}, Lemma \ref{cor 3.3} replaces Corollary 3.3 of \cite{vertodd}, Lemma \ref{lemma 3.4} replaces Lemma 3.4 of \cite{vertodd}, and Lemma \ref{lemma 3.5} replaces Lemma 3.5 of \cite{vertodd}, and thus, we do not include it here.
%We should note the solvability of $G$ is only needed in order %to apply Lemma \ref{lemma 3.5}.  If we can remove the solvable %hypothesis there, we can remove it here also.

With Theorem \ref{second main}, we can prove analogs of the results in \cite{rmjm} when $2 \in \pi$.  For example, this next result is the analog of Theorem 1.1 of \cite{rmjm}.  The proof of this next result is essentially identical to the proof of Theorem 1.1 of \cite{rmjm}.

\begin{lemma}
Suppose $G$ is $\pi$-separable and $2 \in \pi$.  Let $\chi \in \irrg$ satisfy $\chi^o \in \ipig$.  If $N$ is a normal subgroup of $G$ and $\psi$ is an irreducible constituent of $\chi_N$, then there is a generalized vertex $(Q,\delta)$ for $\chi$ so that $(Q \cap N,\delta_{Q \cap N})$ is a generalized vertex for $\psi$.
\end{lemma}

\section{Lifts}

%The purpose of this section of notes is to give the proof that
%In a different set of notes, we'll see that a very similar %argument (modulo one particular hard lemma) allows us to %remove the abelian hypothesis, and show that the number of %lifts of $\varphi$ is bounded above by $|Q : Q'|$.

Throughout, $G$ is a $\pi$-separable group and $\pi$ is a set of primes with $2 \in \pi$.  We will also use $L_{\varphi}$ and $L_{\varphi} (Q, \delta)$ to denote the set of lifts of $\varphi \in \ipig$ and the set of lifts of $\varphi \in \ipig$ with generalized vertex $(Q, \delta)$, respectively.

In this section, we prove several results about lifts with a specified generalized vertex.  In this first lemma, we characterize the factors when such a character has a restriction to a normal subgroup whose irreducible constituents are $\pi$-factored.

\begin{lemma} \label{factor}
Suppose $G$ is a $\pi$-separable group with $2 \in \pi$.  Let $\chi \in \irrg$ with $\chi^o \in \ipig$ and let $(Q,\delta)$ be a generalized vertex for $\chi^o$.  If $N$ is a normal subgroup of $G$ such that $\chi_N$ has $\pi$-factored irreducible constituents, then $Q \cap N$ is a Hall $\pi$-complement of $N$, and there exists a $\pi$-special character $\alpha$ and a $\pi'$-special character $\beta$ in $\irr(N)$ so that $\alpha\beta$ is a constituent of $\chi_N$, $\alpha^o$ is a constituent of $(\chi^o)_N$, and $\beta_{Q \cap N} = \delta_{Q \cap N}$.
\end{lemma}

%We need the normal nucleus in this proof, I think.  Should we define it somewhere previous?

\begin{proof}
Let $V$ be the subgroup of a normal nucleus pair that has $(Q,\delta)$ as a normal vertex.  We know that $V$ contains the largest normal subgroup of $G$ where the restriction of $\chi$ has $\pi$-factored constituents, so $N \le V$.  Since $Q$ is a Hall $\pi$-complement of $V$, it follows that $Q \cap N$ is a Hall $\pi$-complement of $N$.

We have $\gamma \epsilon \in \irr(V)$ where $(\gamma \epsilon)^G = \chi$ and $\gamma$ is $\pi$-special and $\epsilon$ is $\pi'$-special.  We know that $\epsilon_Q = \delta$.  This implies that $\epsilon$ is linear since $\delta$ is linear by Lemma \ref{lemma 2.3}.  Let $\alpha$ be an irreducible constituent of $\gamma_N$, and observe that $\beta = \epsilon_N$ is irreducible.  It follows that $\alpha\beta$ is a constituent of $\chi_N$.  Since $2 \in \pi$, we see that $\beta^o = 1_N$, and so, $\alpha^o$ is a constituent of $\chi^o$.  Finally, $\beta_{Q \cap N} = (\epsilon_N)_{Q \cap N} = (\epsilon_Q)_{Q \cap N} = \delta_{Q \cap N}$.
\end{proof}

This next lemma is related to both Proposition 2.7 and Corollary 2.8 of \cite{pisep}.

\begin{lemma} \label{chief}
Let $G$ be a $\pi$-separable group and suppose that $N$ is normal in $G$ so that $G/N$ is a $\pi$-group.  Let $\nu \in \irr(N)$ so that $\nu = \alpha\beta$ where $\alpha$ is $\pi$-special and $\beta$ is $\pi'$-special.
Then the following are equivalent.
\begin{enumerate}
\item $\beta$ is $G$-invariant.
\item All irreducible constituents of $\nu^G$ are $\pi$-factored.
\item Some irreducible constituent of $\nu^G$ is $\pi$-factored.
\end{enumerate}
\end{lemma}

\begin{proof}
Suppose first that $\beta$ is $G$-invariant.  We know that all irreducible constituents of $\alpha^G$ are $\pi$-special.  Since $|G:N|$ is a $\pi$-number, $\beta$ is $\pi'$-special, and $\beta$ is $G$-invariant, we see that $\beta$ has a $\pi'$-special extension $\hat\beta \in \irrg$.  We know that $\nu^G = (\alpha\beta)^G = \alpha^G \hat\beta$.  We conclude that all the irreducible constituents of $\nu^G$ are the products of irreducible constituents of $\alpha^G$ with $\hat\beta$ and thus they are $\pi$-factored.

Obviously, if all irreducible constituents of $\nu^G$ are $\pi$-factored, then some irreducible constituent of $\nu^G$ is $\pi$-factored.

Suppose some irreducible constituent of $\nu^G$ is $\pi$-factored.  Let $\gamma \delta$ be an irreducible factored constituent of $\nu^G$ where $\gamma$ is $\pi$-special and $\delta$ is $\pi'$-special.  The irreducible constituents of $\gamma_N$ are $\pi$-special, and $\delta_N \in \irr(N)$ is $\pi'$-special.  Hence, the constituents of $(\gamma \delta)_N = \gamma_N \delta_N$ are $\pi$-factored.  By Frobenius reciprocity, $\nu = \alpha \beta$ is a constituent of $\gamma_N \delta_N$.  By the uniqueness of factorization, we conclude that $\delta_N = \beta$, and $\beta$ is $G$-invariant.
\end{proof}

This next result shows that if one character in $L_{\varphi} (Q,\delta)$ has a restriction to a normal subgroup whose irreducible constituents are factored, then all characters in the set $L_{\varphi}(Q, \delta)$ have that property.  We will not explicitly need this fact, but implicitly, we feel that it justifies the inclusion in Lemmas \ref{nicase} and \ref{pidashcase} of the hypothesis that the characters in $L_{\varphi} (Q,\delta)$ have $\pi$-factored constituents when restricted to a normal subgroup $N$.

\begin{lemma} \label{norm fact}
Let $G$ be a $\pi$-separable group, and assume $2 \in \pi$.  Let $\varphi \in \ipig$ have vertex $Q$.  Suppose that $\chi, \psi \in L_{\varphi} (Q,\delta)$.  If $N$ is a normal subgroup of $G$, then $\chi_N$ has $\pi$-factored irreducible constituents if and only if $\psi_N$ has $\pi$-factored irreducible constituents.
\end{lemma}

\begin{proof}
Since the hypotheses in $\chi$ and $\psi$ are symmetric, it suffices to show that if the irreducible constituents of $\chi_N$ are $\pi$-factored, then the irreducible constituents of $\psi_N$ are $\pi$-factored.  Let $L$ be maximal subject to the conditions that $L$ is normal in $G$, $L \leq N$, and the constituents of $\psi_L$ are $\pi$-factored.

If $L = N$, then the result holds.  Hence, we assume that $L < N$, and we find a contradiction.  Observe that $\chi_L$ will have $\pi$-factored irreducible constituents.  By Lemma \ref{factor}, $\chi_L$ has $\alpha_1 \beta_1$ and $\psi_L$ has $\alpha_2 \beta_2$ as irreducible constituents where the $\alpha_i$ are $\pi$-special and $\alpha_i^o$ are constituents of $\varphi_L$ and the $\beta_i$ are $\pi'$-special with $\beta_i|_{Q \cap N} = \delta_{Q \cap N}$.  Since $Q \cap N$ is a Hall $\pi$-complement of $N$, we know that restriction is an injection from the $\pi'$-special characters of $N$ into $\irr(Q \cap N)$.  It follows that $\beta_1 = \beta_2$.  We will write $\beta$ for $\beta_i$ from now on.  Observe that $\alpha_1^o$ and $\alpha_2^o$ are $G$-conjugate, so $\alpha_1$ and $\alpha_2$ are $G$-conjugate.

Let $L < M \le N$ be chosen so that $M$ is normal in $G$ and $M/L$ is a chief factor for $G$.  We know that $M/L$ is either a $\pi$-group or a $\pi'$-group.  Also, some irreducible constituent of $\chi_M$ must be an irreducible constituent of $(\alpha_1 \beta)^M$.  Since the irreducible constituents of $\chi_M$ are $\pi$-factored, we conclude that some irreducible constituent of $(\alpha_1 \beta)^M$ is $\pi$-factored.  If $M/L$ is a $\pi$-group, then Lemma \ref{chief} implies that $\beta$ is $M$-invariant.  Applying Lemma \ref{chief} again, we see that all the irreducible constituents of $(\alpha_2 \beta)^M$ are $\pi$-factored.  Since some irreducible constituent of $\psi_M$ is a constituent of $(\alpha_2 \beta)^M$, we conclude that $\psi_M$ has $\pi$-factored irreducible constituents, and this violates the maximality of $L$.

Thus we may assume that $M/L$ is a $\pi'$-group.  By Lemma \ref{chief}, $\alpha_1$ is $M$-invariant.  We know that there exists $g \in G$ so that $\alpha_2 = \alpha_1^g$.  So $\alpha_1 \beta^{g^{-1}}$ is $G$-conjugate to $\alpha_2 \beta$.  By Lemma \ref{chief}, the irreducible constituents of $(\alpha_1 \beta^{g^{-1}})^M$ are all $\pi$-factored.  Observe that $\psi_M$ must share an irreducible constituent with $(\alpha_1 \beta^{g^{-1}})^M$, and so, the irreducible constituents of $\psi_M$ are $\pi$-factored and this violates the maximality of $L$.  This contradiction proves the result.
\end{proof}

We will not use this next corollary in this paper, but we think it is interesting in its own right.

\begin{corollary} \label{kernels}
Let $G$ be a $\pi$-separable group with $2 \in \pi$.  Let $\varphi \in \ipig$ have vertex $Q$.  If $\chi, \psi \in L_{\varphi} (Q,\delta)$, then $\ker {\chi} = \ker {\psi}$.
\end{corollary}

\begin{proof}
Given the symmetry between $\chi$ and $\psi$, it suffices to prove that $\ker {\chi} \le \ker {\psi}$.  Let $K = \ker {\chi}$.  Since the irreducible constituents of $\chi_K$ are $1_K$, the irreducible constituents of $\chi_K$ are $\pi$-factored.  By Lemma \ref{norm fact}, the irreducible constituents of $\psi_K$ are $\pi$-factored.  Observe that $1_K^o$ is the irreducible constituent of $\varphi_K$ and $1_{Q \cap K} = \delta_{Q \cap K}$.  By Lemma \ref{factor}, we know that $\psi_K$ has an irreducible constituent $\alpha \beta$ where $\alpha$ is $\pi$-special and $\alpha^o$ is a constituent of $\varphi_K$ and $\beta$ is $\pi'$-special and $\beta_{Q \cap K} = \delta_{Q \cap K}$.  It follows that $\alpha^o = 1_K^o$ which implies that $\alpha = 1_K$ and $\beta_{Q \cap K} = 1_{Q \cap K}$ which implies that $\beta = 1_K$.  We conclude that $1_K$ is a constituent of $\psi_K$ and thus $K \le \ker {\psi}$.
\end{proof}

\section{Counting Lifts}

We now work to prove that if $G$ is a solvable group and $\pi$ is a set of primes with $2 \in \pi$, then the number of lifts of a character $\varphi \in \ipig$ with {\it{abelian}} vertex subgroup $Q$ is bounded above by $|Q|$.  When $Q$ is contained in ${\bf O}_{\pi \pi'} (G)$, we prove that the number of lifts of $\varphi$ is at most $|Q:Q'|$.

Our argument is essentially an inductive argument.  In this next result, we consider one of the cases that arises.

\begin{lemma} \label{nicase}
Let $G$ be a $\pi$-separable group, and assume $2 \in \pi$.
Suppose $\varphi \in \ipig$ has vertex subgroup $Q$, and $N$ is normal in $G$ with $\alpha \in \ipi(N)$ a constituent of $\varphi_N$ such that the Clifford correspondent $\xi$ of $\varphi$ in $\ipi(G_{\alpha} \mid \alpha)$ has vertex subgroup $Q$, and write $T = G_{\alpha}$.  Fix a character $\delta \in \irr(Q)$. Let $\delta_1, \delta_2, \ldots, \delta_k$ denote a set of representatives of the $\ntq$ orbits of the action of $\ntq$ on the $\ngq$ orbit containing $\delta$.  Suppose the characters in $L_{\varphi} (Q,\delta)$ have $\pi$-factored constituents when restricted to $N$.  Then induction is a bijection from $\bigcup L_{\xi} (Q, \delta_i)$ to $L_{\varphi} (Q, \delta).$
\end{lemma}

\begin{proof}
Using the Clifford correspondence for $\pi$-partial characters, we know that $\xi^G = \varphi$.  If $\theta \in \bigcup L_{\xi} (Q, \delta_i)$, then $(\theta^G)^o = (\theta^o)^G = (\xi)^G = \phi$.  Thus, induction is a map from $\bigcup L_{\xi} (Q, \delta_i)$ to $L_{\varphi}$.  Since $\theta \in L_{\xi} (Q, \delta_i)$, it follows from Theorem \ref{second main}that $\theta^G$ has vertex $(Q, \delta_i)$, which is $G$-conjugate to $(Q, \delta)$, and thus, induction is a map from $\bigcup L_{\xi} (Q, \delta_i)$ to $L_{\varphi} (Q, \delta)$.

To see that induction is surjective, suppose $\chi \in
L_{\varphi} (Q, \delta)$.  By Lemma \ref{factor}, $\chi_N$ has an irreducible constituent $\gamma \beta$ where $\gamma$ is $\pi$-special and $\beta$ is $\pi'$-special and $\gamma$ lifts $\alpha$.  Thus, $\chi$ is induced irreducibly from some character $\omega \in \irr(G_{\gamma \beta} \mid \gamma \beta)$, and note that $G_{\gamma \beta} \subseteq G_{\alpha}$ and $\omega^o$ must induce to $\xi$.  Let $\mu = \omega^T$, so $\mu^o = \xi$, and $\mu^G = \chi$.  Applying Theorem \ref{second main}, $\mu$ has a vertex that is $G$-conjugate to $(Q, \delta)$, so $\mu$ has a vertex that is $T$-conjugate to one of the $(Q, \delta_i)$. Thus, $\mu \in \bigcup L_{\xi}(Q, \delta_i)$ for some character $\delta_i$.

Finally, we show that induction is injective.  Suppose we have $\mu_1, \mu_2 \in \bigcup L_{\xi} (Q, \delta_i)$ with $\mu_1^G = \mu_2^G$.  Let $\chi = \mu_1^G = \mu_2^G$.  Notice that $(\mu_j^o)_N = \xi_N = e \alpha$ for some integer $e$. Let $\gamma$ be $\pi$-special and lift $\alpha$.  By Lemma \ref{factor}, we can find $\pi'$-special characters $\beta_1$ and $\beta_2$ so that $\gamma \beta_j$ is a constituent of $(\mu_j)_N$.  Notice that this implies that $\gamma \beta_1$ and $\gamma \beta_2$ are both constituents of $\chi_N$. It follows that there exists $g \in G$ so that $(\gamma \beta_2)^g = \gamma \beta_1$.  This implies that $\gamma^g = \gamma$ and $\beta_2^g = \beta_1$.  It follows that $\alpha^g = \alpha$ and thus $g \in T$.  This implies that $\gamma \beta_1$ is a constituent of $(\mu_2)_N$.  We can find $\omega_1, \omega_2 \in \irr(G_{\gamma \beta_1} \mid \gamma \beta_1)$ so that $\omega_j^T = \mu_j$.  We now have $(\omega_j)^G = \chi$, and since Clifford induction is a bijection, we conclude that $\omega_1 = \omega_2$ and thus $\mu_1 = \mu_2$.
\end{proof}

We will not actually need that the induction map in Lemma \ref{nicase} is injective, only that it is surjective.  The fact that the map is injective implies that the union $\bigcup L_{\xi} (Q, \delta_i)^G$ is disjoint, though again, we do not need that.  The next lemma also considers a case that arises in the inductive argument.

\begin{lemma} \label{pidashcase}
Let $G$ be a $\pi$-separable group, and assume $2 \in \pi$.  Suppose $\varphi \in \ipig$ has vertex subgroup $Q$, and $N$ is normal in $G$ so that $\varphi_N$ is homogeneous.  Let $(Q,\delta)$ be a generalized vertex for a lift of $\varphi$.  Suppose the characters in $L_{\varphi} (Q,\delta)$ have $\pi$-factored constituents when restricted to $N$.  Write $\beta \in \irr(N)$ for the $\pi'$-special character so that $\beta_{Q \cap N} = \delta_{Q \cap N}$.  Let $T$ be the stabilizer of $\beta$ in $G$. Let $\zeta_1, \dots, \zeta_k$
be the distinct $\pi$-partial characters in $\ipi(T)$ that induce $\varphi$ and have $Q$ as a vertex.  Then induction is a bijection from $\bigcup L_{\zeta_i} (Q, \delta)$ to $L_{\varphi} (Q,\delta)$.
\end{lemma}

\begin{proof}
Let $\alpha$ be the $\pi$-special character in $\irr(N)$ so that $\alpha^o$ is the unique irreducible constituent of $\varphi_N$.  Observe that $\alpha^o$ will be the unique irreducible constituent of $(\zeta_i)_N$ for all $i$.  Suppose $\theta \in L_{\zeta_i} (Q,\delta)$.  Then $(\theta^G)^o = (\theta^o)^G = (\zeta_i)^G = \varphi$, and thus, $\theta^G \in L_{\varphi}$.  Note that $(Q,\delta)$ will be a generalized vertex of $\theta^G$, and so, induction is a map from $\bigcup L_{\zeta_i} (Q,\delta)$ to $L_{\varphi} (Q,\delta)$.

Suppose $\chi \in L_{\varphi} (Q,\delta)$.  By Lemma \ref{factor}, we see that $\alpha \beta$ is an irreducible constituent of $\chi_N$.  Observe that $\alpha$ is $G$-invariant, since $\alpha^o$ is $G$-invariant.  Thus, $T$ is the stabilizer of $\alpha \beta$ in $G$.  We can find $\theta \in \irr(T \mid \alpha \beta)$ so that $\theta^G = \chi$.  Since $(\theta^o)^G = (\theta^G)^o = \chi^o = \varphi$, it follows that $\theta \in \irr(T)$ is a lift of a character in $\ipi(T)$ that induces $\varphi$.

Let $(Q^*,\delta^*)$ be a generalized vertex for $\theta$.  By Lemma \ref{factor}, we know that $\theta_N$ has an irreducible
constituent $\alpha \beta^*$ where $\beta^*$ is $\pi'$-special in $\irr(N)$ so that $(\beta^*)_{Q^* \cap N} = \delta^*_{Q^* \cap N}$.  Observe that $\alpha\beta$ is the unique irreducible constituent of $\theta_N$.  Hence, we have $\alpha\beta = \alpha \beta^*$.  The uniqueness of factorization implies that $\beta = \beta^*$.

Since $(Q,\delta)$ and $(Q^*,\delta^*)$ are both generalized
vertices for $\chi$, we may apply Theorem \ref{second main} to find an element $g \in G$, so that $(Q^*,\delta^*) = (Q,\delta)^g$.  We see that $\beta^g_{Q^g \cap N} = \delta^*_{Q^g \cap N} = \beta_{Q^g \cap N}$.  By Lemma \ref{factor}, $Q^g \cap N$ is a Hall $\pi$-complement of $N$.  We know that restriction is an injection from the $\pi'$-special characters of $N$ to $\irr(Q^g \cap N)$. We
conclude that $\beta^g = \beta$.  It follows that $g \in T$, and hence, $Q \le T$, and $(Q,\delta)$ is a generalized vertex for $\theta$.  Notice that $Q$ is now a vertex for $\theta^o$, and so, $\theta^o = \zeta_i$ for some $i$.  It follows that induction is a surjective map from $\bigcup L_{\zeta_i} (Q,\delta)$ onto $L_{\varphi} (Q,\delta)$.

Finally, notice that by Lemma \ref{factor} the elements of $\bigcup L_{\zeta_i} (Q,\delta)$ all lie in $\irr(T \mid \alpha \beta)$.  We now apply Clifford's theorem to see the map is injective.
\end{proof}

We now state and prove the result that forms the main part of our argument.  In particular, we find a bound on the number of lifts of $\varphi$ that have generalized vertex $(Q,\delta)$.

\begin{theorem} \label{avhardpart}
Assume $G$ is a $\pi$-separable group and $2 \in \pi$.
Suppose that $\varphi \in \ipig$ has vertex subgroup $Q$, and let $\delta \in \irr(Q)$.  Assume either $Q$ is abelian or $Q \le {\bf O}_{\pi \pi'} (G)$.  Then $|L_{\varphi} (Q, \delta)| \leq |\ngq : \norm G{Q, \delta}|$.
\end{theorem}

\begin{proof}
We work by induction on $|G|$.  If $L_{\varphi} (Q,\delta)$ is empty, then the result is trivial.  Thus, we assume that $L_{\varphi} (Q,\delta)$ is nonempty.

Let $N$ be any normal subgroup of $G$ so that the restrictions of characters in $L_{\varphi} (Q, \delta)$ to $N$ have $\pi$-factored irreducible constituents.  Choose $\eta \in \ipi(N)$ such that $\eta$ is a constituent of $\varphi_N$.

If $G = N$, then the characters in $L_{\varphi} (Q,\delta)$ are $\pi$-factored with linear $\pi'$-special factor.  Let $\theta$ be $\pi$-special in $\irrg$ so that $\theta^o = \eta = \varphi$.  Notice that $\theta$ must be the $\pi$-special factor of the characters in $L_{\varphi} (Q,\delta)$. Also, $Q$ is a Hall $\pi$-complement of $G$, and so there is a unique $\pi'$-special extension $\beta \in \irrg$ of $\delta$. We conclude that $\theta \beta$ is the only character in $L_{\varphi} (Q,\delta)$, and the result holds.

%Let $N = \opig$, and choose $\alpha \in \Ipi N = \irr N $ such %that $\alpha$ lies under $\varphi$
Suppose now that $N < G$.  We can find a Clifford correspondent $\xi$ of $\varphi$ in $\ipi(G_{\eta} \mid \eta)$ that has vertex subgroup $Q$. Write $T = G_{\eta}$, and note that $N$ and $\eta$ satisfy the conditions of Lemma \ref{nicase}.  Assume that $T < G$.  By Lemma \ref{nicase}, we see that $L_{\varphi}(Q, \delta) = \bigcup L_{\xi}(Q, \delta_i)^G,$ and thus
$$|L_{\varphi}(Q, \delta)| \leq \sum_{i =1}^{k} |L_{\xi}(Q,
\delta_i)^G| \leq \sum_{i=1}^k |L_{\xi}(Q, \delta_i)| \leq
\sum_{i=1}^k |\ntq : \norm T{Q, \delta_i}|,$$
where the last inequality follows by induction.

However, $|\ntq : \norm T{Q, \delta_i}|$ is precisely the orbit size of the $\ntq$ orbit containing $\delta_i$ in the action of $\ntq$ on the $\ngq$ orbit containing $\delta$.  Thus, the sum $$\sum_{i=1}^k |\ntq : \norm T{Q, \delta_i}|$$ is precisely $|\ngq : \norm G{Q, \delta}|$, and we are done.  We may assume that $\eta$ is invariant in $G$.

Now, if we consider $\opig$, then it is trivial to see that the restrictions to $\opig$ of characters in $L_{\varphi} (Q,\delta)$ are $\pi$-factored.  Let $M =
{\bf O}_{\pi \pi'} (G)$, and let $P$ be a Hall $\pi'$-subgroup of $M$, so that $M = P \opig$.  Let $\alpha \in \irr(\opig) = \ipi(\opig)$ be a constituent of $\varphi_{\opig}$.  Using the previous paragraph with $N = \opig$ and $\eta = \alpha$, we may assume that $\alpha$ is $G$-invariant.  Thus, the constituents of $\alpha^M$ are $\pi$-factored by Lemma \ref{chief}.  Applying Lemma \ref{factor}, we may assume that $P = Q \cap M$.  If $Q$ is abelian, we see that $Q\opig/\opig \subseteq \cent {G/\opig}{M/\opig} \subseteq M/\opig$, where the last containment is by Hall-Higman's Lemma 1.2.3, and thus $Q \le M$.  Thus, in all cases, $Q \le M$, and this implies that $P = Q$.

Notice that a Frattini argument shows that $G = M \ngp$.  Let
$\hat\alpha$ be the unique $\pi$-special extension of $\alpha$ to $M$, and note that since $\alpha$ is invariant in $G$, then $\hat\alpha$ is invariant in $G$.  Let $\hat{\delta}$ be the unique $\pi'$-special extension of $\delta$ to $M$.   We set $I = G_{\hat{\delta}}$.  Assume that $I < G$.  Let $\zeta_1, \zeta_2, \ldots, \zeta_m$ denote the characters of $\ipi(I)$ with vertex $Q$ that induce irreducibly to $\varphi$.  Note that $m \leq |G : I| = |\ngq : \norm IQ|$.

By Lemma \ref{pidashcase}, induction is a bijection from $\displaystyle \bigcup_{i=1}^{m} L_{\zeta_i}(Q, \delta)$ to $L_{\varphi}(Q, \delta)$.  Therefore, by the inductive hypothesis (since we are assuming that $I < G$) we see that $$|L_{\varphi}(Q, \delta)| = \sum_{i=1}^{m} |L_{\zeta_i}(Q, \delta)| \leq \sum_{i=1}^m |\norm IQ : \norm I{Q, \delta}| = m|\norm IQ : \norm I{Q, \delta}|.$$  Now, clearly $\norm I{Q, \delta} \subseteq \norm G{Q, \delta}$.  Also, if $x \in G$ normalizes $Q$ and stabilizes $\delta$, then $x$ fixes $\hat{\delta}$ and therefore $x$ fixes $\hat{\alpha} \hat{\delta}$ (recall $\hat{\alpha}$ is invariant in $G$), and thus $x \in I$. Thus $\norm G{Q, \delta} = \norm I{Q, \delta}$.

The above inequality becomes $$|L_{\varphi}(Q, \delta)| \leq m|\norm IQ : \norm G{Q, \delta}|$$ and since $m \leq |\ngq : \norm IQ|$, we obtain $$|L_{\varphi}(Q, \delta)| \leq |\ngq : \norm IQ| |\norm IQ : \norm G{Q, \delta}| = |\norm GQ : \norm G{Q, \delta}|$$ and we are done in the case that $I < G$.

We may assume that $\hat{\delta}$ is invariant in $G$. Let $L/M = \opi(G/M)$.
%(Of course, if $M = G$ the result is obvious, so we may assume %that $M < G$ and thus that $M < L$ since $M/N = \opid(G/N)$.)
Note that since $\hat{\delta}$ is invariant in $G$, the constituents of $(\hat\alpha\hat\delta)^L$ are $\pi$-factored. Also, $\hat{\delta}$ extends to a unique $\pi'$-special character $\delta^* \in \irr(L)$ that is also invariant in $G$. It follows that if $\chi \in L_{\varphi}(Q, \delta)$, then the constituents of $\chi_L$ are factorable.  By Lemma \ref{factor}, $\chi_L$ has an irreducible constituent $\alpha^* \delta^* \in \irr(L)$, where $\alpha^*$ lies over $\hat{\alpha}$ and is $\pi$-special.  Taking $L = N$ in the above work, we see that $\alpha^*$ is $G$-invariant.
%(this follows since $\delta^*$ is invariant in $G$).
%Assume that $\alpha^*$ is not invariant in $G$, and let $J = %G_{\alpha^*}$, and note that by conjugating if necessary, we %may assume that the Clifford correspondent $\nu$ for $\varphi$ %in $\Ipi {J \mid (\alpha^*)^o}$
%has vertex subgroup $Q$.  Thus we may apply Lemma \ref{nicase} %to $L$ and $(\alpha^*)^o$ to see that $$L_{\varphi}(Q, \delta) %= \bigcup L_{\nu}(Q, \delta_j)^G,$$ where $\delta_1, \ldots, %\delta_l$ is a set of representatives of the $\norm JQ$ orbits %of the action of $\norm JQ$ on the $\ngq$ orbit containing %$\delta$.  Therefore by induction we see that %$$|L_{\varphi}(Q, \delta)| \leq \sum_{j=1}^{l} |L_{\nu}(Q, %\delta_j)^G| \leq \sum_{j=1}^l |L_{\nu}(Q, \delta_j)| \leq %\sum_{j=1}^{l} |\norm TQ : \norm T{Q, \delta_j}| \newline =$$ %$$ |\ngq : \norm G{Q, \delta}|$$ and we are done in this %case.

Let $K/L = \opid(G/L)$.  Then $\alpha^*$ and $\delta^*$ are invariant in $G$, and thus there is a $\pi$-factorable character in $K$ lying over $\alpha^* \delta^*$.  If $K > L$, this contradicts the fact that $(Q, \delta)$ is a vertex (because the vertex subgroup would be larger if $K > L$).  Therefore $K = L$ and $\opid (G/L)$ is trivial, meaning $G = L$ and every character in $L_{\varphi}(Q, \delta)$ is factorable, and we have seen that we are done in this case.
%Thus $\alpha^*$ is the unique $\pi$-special lift of $\varphi$, %and since every character in $L_{\varphi}(Q, \delta)$ lies %over $\hat{\alpha} \hat{\delta}$, then every character in %$L_{\varphi}(Q, \delta)$ is factorable, and thus %$|L_{\varphi}(Q, \delta)| = 1$ and we are done.
\end{proof}

We may now easily prove our main result.

\begin{theorem} \label{mainabelian}
Assume $G$ is $\pi$-separable and $2 \in \pi$.  Suppose that $\varphi \in \ipig$ has vertex subgroup $Q$ where either $Q$ is abelian or $Q \le {\bf O}_{\pi \pi'} (G)$.  Then $|L_{\varphi}| \leq |Q:Q'|$.
\end{theorem}

\begin{proof}
Since $G$ is $\pi$-separable and $2 \in \pi$, Lemma \ref{lemma 2.3} implies that any vertex character of a lift of $\varphi$ is a linear character of $Q$.  Thus, if $\delta_1, \ldots, \delta_k$ are representatives of the $\ngq$ orbits of the linear characters of $Q$, then by Theorem \ref{avhardpart},
$$|L_{\varphi}| = \sum_{i =1}^{k} |L_{\varphi}(Q, \delta_i)| \leq \sum_{i=1}^k |\ngq : \norm G{Q, \delta_i}| = |Q:Q'|.$$
\end{proof}

We would like to remove the hypothesis on $Q$.  We now discuss one potential way to do this.  Observe that in the proof of Theorem \ref{avhardpart}, we can take $N$ to be maximal so that $N$ is normal and the restrictions of characters in $L_{\varphi} (Q,\delta)$ have factored irreducible constituents.  As in that proof, we may assume that $\eta$ is $G$-invariant where $\eta$ is an irreducible constituent of $\varphi_N$.  Let $\alpha$ be the $\pi$-special character in $N$ that lifts $\eta$, and observe that $\alpha$ is $G$-invariant.  Let $\beta \in \irr(N)$ be the $\pi'$-special character so that $\beta_{Q \cap N} = \delta_{Q \cap N}$.  If we could show that $\beta$ has to be $G$-invariant, we could apply Lemma \ref{chief} to see that $N = G$, and we would be done as in the third paragraph of that proof.

If $\beta$ is not $G$-invariant, we take $I$ to be the stabilizer of $\beta$ in $G$.  Using the same argument as in that proof, we can show that $|L_{\varphi} (Q,\delta)| \le m |\norm IQ:\norm G{Q,\delta}|$ where $m$ is the number of characters in $\ipi(I)$ that have vertex $Q$ and induce $\varphi$.  If $m \le |\norm GQ: \norm IQ|$, then the desired conclusion would follow.  As we can see in the proof, we chose hypotheses on $Q$ that guarantee that can find a subgroup $I$ where this happens.  However, we ask whether this happens for arbitrary $I$.  In particular, if $\varphi \in \ipig$ has vertex $Q$ and $I$ is a subgroup that contains $Q$, then is it true that the number of $\pi$-partial characters of $I$ with vertex $Q$ that induce $\varphi$ is at most $|\norm GQ:\norm IQ|$?  The second author has shown that this is true when $|G|$ is odd or $2 \not\in \pi$ \cite{lewisnotes}, but we have not been able to settle the question when $2 \in \pi$.

\end{document}